\documentclass[final,leqno,onefignum,onetabnum]{siamltex}
\usepackage{amsmath}
\usepackage{amsfonts}
\usepackage{mathrsfs}
\usepackage{amssymb}
\usepackage{verbatim}
\usepackage{url}
\usepackage{enumerate}
\usepackage{graphicx,color}
\usepackage{alltt}
\usepackage[margin=1.4in]{geometry}

\usepackage{calc}  
\usepackage{amsfonts}
\usepackage[colorlinks=true, bookmarksopen,
                      pdfcreator={pdftex},
            pdfsubject={algorithms},
            linkcolor={blue},
            anchorcolor={black},
            citecolor={red},
            filecolor={magenta},
            menucolor={black},
            pagecolor={red},
            plainpages=false,pdfpagelabels,
            urlcolor={blue}]{hyperref}
\let\oldeq\equation{}\def\equation{\par\vspace{-\parskip}\oldeq}
\newcounter{saveeqn}%

\author{
Vandana Sharma\footnotemark[1]     }

\begin{document}
\title{Global Existence and Uniform Estimates For Solutions to Reaction Diffusion Systems with Mass Transport Type Boundary Conditions}
\maketitle

\renewcommand{\thefootnote}{\fnsymbol{footnote}}

 \footnotetext[2]{Department of Mathematics, Indian Institute of Technology Jodhpur, Rajasthan, India, 342037. Email: \textup{\nocorr \texttt{vandanas@iitj.ac.in}}.}

\begin{abstract}
We consider reaction diffusion systems where components diffuse inside the domain and react on the surface through mass transport type boundary conditions. Under reasonable hypotheses, we
establish the existence of component wise non-negative global solutions which are uniformly bounded in the sup norm. \end{abstract}

\begin{keywords}
 reaction-diffusion equations, mass transport, conservation of mass, global existence.
\end{keywords}

\begin{AMS}
35K57, 35B45
\end{AMS}

\section{Introduction} 

Suppose $m\ge 2$ is a natural number, $T>0$, and $\Omega$ is a bounded domain in $\mathbb{R}^n$ with smooth boundary M ($\partial \Omega$) belonging to the class $C^{2+\sigma}$ with $\sigma>0$, such that $\Omega$ lies locally on one side of its  boundary. $\eta$ is the unit outward normal to $M$ (from $\Omega$), and $\Delta$ is the Laplace operator. We are interested in the system 
\begin{align}\label{primary}
\frac{\partial u_i}{\partial t}&=d_i \Delta u_i+F_i(u)  \ & (x,t)\in \Omega\times(0,T)&\  \text{for} \  i=1, ..., m\nonumber\\
d_i\frac{\partial u_i}{\partial\eta}&=G_i(u)  \ & (x,t)\in M\times(0,T)&\  \text{for}\  i=1, ..., m\\
u_i&=w_i  \ & (x,t)\in \overline\Omega\times \left\{0\right\}&\ \text{for}\ i=1,...,m.  \nonumber
\end{align}
Here $d_i>0$ for all $i=1, ...,m$, $F=(F_i), G=(G_i) : \mathbb{R}^m \rightarrow \mathbb{R}^m$ are smooth, quasi positive and polynomially bounded, and the initial data $w=(w_i) \in C^2(\overline{\Omega})$ with $w_i\ge 0$ for all $i=1,...,m$, and
\[ d_i \frac{\partial w_i}{\partial \eta}=G_i(w)\quad \text{on}\quad M\quad \text{for all}\quad i=1,...,m.\]For those not familiar with quasi positivity, see assumption ($V_{\text{QP}}$) in the next section.

In 1987, Hollis, Martin and Pierre \cite{RefWorks:5} considered (\ref{primary}) in the case when $m=2$ and $G_1(u)=G_2(u)=0 $. The conditions on the vector field $F(u)$ above guarantee local well posedness of nonnegative solutions, and the authors asked whether solutions would exist globally if there exist constants $a>0$ and $K\in \mathbb{R}$ such that
\begin{align}\label{martinmass}
aF_1(u)+F_2(u)\le K(u_1+u_2+1)
\end{align}
for all $u_1,u_2\ge 0$. 
The assumption (\ref{martinmass}) easily implies bounds for $\|u_i(\cdot,t)\|_{1,\Omega}$ for $i=1,2$, and more importantly, in the absence of diffusion, this assumption implies  solutions exist globally, by adding $a$ times the differential equation for $u_1$ to the differential equation for $u_2$. In the case when $d_1,d_2>0$, Hollis et al proved (\ref{martinmass}) implies that the solutions to (\ref{primary}) are global if at least one of $\|u_1\|_{\infty}$ or $\|u_2\|_{\infty}$ is a priori bounded on $\Omega\times(0,T)$ for every $T>0$. The latter assumption is not easily removed, since Pierre and Schmitt  \cite{RefWorks:104} gave an example of a system that satisfies the assumptions above, and blows up in finite time. Although the particular example had Dirichlet boundary conditions, as opposed to the homogeneous Neumann boundary conditions being considered, it seemed clear that adjustments could be made to create a system for which (\ref{martinmass}) holds, and the solution blows up in finite time.

It's less obvious that (\ref{martinmass}) also implies bounds for $\|u_i\|_{2,\Omega\times(0,T)}$ for $i=1,2$ and $T>0$ cf. \cite{RefWorks:3}, and more recently, for $\|u_i\|_{2+\epsilon,\Omega\times(0,T)}$ for $i=1,2$, $T>0$ and $\epsilon>0$ sufficiently small (independent of $T$), \cite{RefWorks:101}. In the past 30 years, there has been an explosion of results for (\ref{primary}), in the setting of $m\ge 2$ and $G_i(u)=0$ for all $i$, with various assumptions mirroring (\ref{martinmass}). These assumptions impose additional structure on the vector field $F(u)$ to obtain results without assuming a priori sup norm bounds on some subset of the components of the solution. \cite{RefWorks:86} contains an excellent history of this problem and a great deal of the subsequent work.

One useful assumption for attacking (\ref{primary}) in the setting when $G_i(u)=0$ for all $i$, is the so-called linear intermediate sum condition, which assumes the existence of an $m\times m$ lower triangular matrix $A=(a_{i,j})$ with positive diagonal entries, and a constant $K\in\mathbb{R}$ so that 
\begin{align}\label{intermed}
AF(u)\le K\vec{1}\left(\sum_{i=1}^m u_i+1\right)
\end{align}
for all $u_i\ge 0$. This assumption was first introduced in \cite{RefWorks:3} to prove global existence, and variants have evolved since that time, including the right hand side of (\ref{intermed}) being squared when $n=2$, in \cite{RefWorks:101}. It has also been shown that when (\ref{intermed}) is not assumed, but only an $m$ component version of (\ref{martinmass}) is assumed, and the vector field $F(u)$ is componentwise quadratically bounded, then solutions exist globally, (cf. \cite{RefWorks:102}, \cite{RefWorks:103}).

Another result in the cased when $G_i(u)=0$ for all $i$, was given in \cite{RefWorks:8}, where the authors showed that global existence could be obtained under the assumption of the existence of a real number $K>0$ so that for every choice of $a=(a_1,...,a_{m-1})$, with $a_1,...,a_{m-1}\ge K$, there exists $L_a\ge 0$ so that 
\begin{align}\label{koua}
\sum_{i=1}^{m-1}a_iF_i(u)+F_m(u)\le L_a\left(\sum_{i=1}^m u_i+1\right)
\end{align}
for all $u_i\ge 0$. Interestingly, this condition makes it possible to create an infinite family of Lyapunov functions that can be used to obtain $L_p$ estimates for every $1<p<\infty$. In the case when $m=2$, it is a simple matter to prove that (\ref{intermed}) is contained in the assumption (\ref{koua}), but for $m>2$, this is not the case. For example, the vector field
\begin{align}\label{nk}
F(u)=\begin{pmatrix}
u_1-u_1u_2u_3 \\
u_1u_2u_3-u_2 \\
u_1u_2u_3-u_3 
\end{pmatrix}
\end{align}
is clearly quasi positive, polynomially bounded, and satisfies (\ref{intermed}) with 
\[ A=\begin{pmatrix} 1 &0&0\\1&1&0\\1&0&1\end{pmatrix}\] 
and $L=1$ . But it does not satisfy (\ref{koua}).

The case of the general system (\ref{primary}), with $G(u)\ne \vec{0}$ has not been extensively explored. The work in \cite{RefWorks:1} proves that a unique, componentwise nonnegative maximal solution to (\ref {primary}) exists on a maximum time interval $(0, T_{\max})$. In addition, if $T_{\max}<\infty$, then the sup norm of $u$ becomes unbounded as $t\rightarrow T_{\max}^{-}$. In this work, we explore two settings. First, we consider (\ref{primary}) in the setting of $m=2$, by asking whether the work in \cite{RefWorks:5} can be extended to the case where $G(u)\ne 0$. More precisely, we ask whether an extension of (\ref{martinmass}) can be used to include the vector field $G(u)$, to prove that  the solution to (\ref{primary}) is global if at least one of $\|u_1\|_{\infty}$ or $\|u_2\|_{\infty}$ is a priori bounded on $\Omega\times(0,T)$ for every $T>0$.  Then we conclude this work by considering (\ref{primary}) in the setting where the assumption (\ref{koua}) is extended to both $F$ and $G$. 

Before leaving this section, we give a handful of conditions on the initial data, and the vector fields $F(u)$ and $G(u)$. The first three of these will be used throughout this work, and various portions of the remaining will be used in our main results. We remark that throughout, $\mathbb{R}^m_{+}$ is the nonnegative orthant in $\mathbb{R}^m$.
\medskip
\begin{enumerate}
\item[($V_{\text{N}}$)]  $w=(w_i)\in C^2(\overline\Omega)$, $w$ is componentwise nonnegative on $\overline\Omega$, and $w$ satisfies the compatibility condition\\
\[ d_i \frac{\partial w_i}{\partial \eta}=G_i(w)\quad on \ M.\]
\item[($V_{\text{F}}$)] $F=(F_i),G=(G_i):\mathbb{R}^m\rightarrow \mathbb{R}^m$ are locally Lipschitz.
\item[($V_{\text{QP}}$)] $F$ and $G$ are {\em quasi positive}.  That is $F_i(u), G_i(u)\geq 0$ for all $u\in\mathbb{R}^m_{+}$ with $u_i=0$ for all $i=1, ..., m$.
\item[($V_{\text{L1}}$)] There exists $b_j>0$ and $L_1\ge 0$ such that \[ \sum_{j=1}^m b_j F_j(z), \sum_{j=1}^m b_j G_j(z)\leq L_1\left(\sum_{j=1}^m z_j +1\right)  \quad  \text{for all} \quad  z\in\mathbb{R}^m_{+}.\] 
\item[($V_{\text{L}}$)] There exists a constant $K>0$, so that if $a=(a_1,...,a_{m-1})$ with $a_1, ...,a_{m-1}\geq K$, and $a_m=1$, then there is a constant $L_a\ge 0$ so that 
\[
\sum_{j=1}^ma_j F_j(z), \sum_{j=1}^m a_jG_j(z)\leq L_a\left(\sum_{j=1}^m z_j+1\right) \quad  \text{for all} \quad  z\in\mathbb{R}^m_{+}. \]
\item[($V_{\text{Poly}}$)] $F$ and $G$ are {\em polynomially bounded}.  That is, there exists $M> 0$ and a natural number $l$ such that \[ \vert F_i(z)\vert, \vert G_i(z)\vert\leq M\left(\sum_{i=1}^m z_i+1\right)^l \ \text{for all}\  z\in \mathbb{R}^m_{+}.\]
\end{enumerate}
Note that $(V_{\text{L}})$ implies $(V_{\text{L1}})$, but the opposite is not true, and we have special need of the value of $L_1$ in $(V_{\text{L1}})$ that holds for this specific case. So we write $(V_{\text{L1}})$ and $V_{\text{L}}$ separately.

The statements of our main results are given in Section 2, and their proofs are given in the remaining sections. 

\section{Notation and Statements of Main Results}
Throughout this work $n\geq 1$. As stated in the introduction, $\Omega$ be a bounded doamin of $\mathbb{R}^n$ with smooth boundary $M$ such that $\Omega$ lies locally on one side of $M$. We define all $L_p$ and Sobolev function spaces on $\Omega$ and $\Omega_T=\Omega\times(0,T)$, and similar definitions can be given on $M$ and $M_T=M\times(0,T)$. Measurability and summability are to be understood everywhere in the sense of Lebesgue.

If $p\ge 1$, then $L_p(\Omega)$ is the Banach space consisting of all measurable functions on $\Omega$ that are $p^{th}$ power summable on $\Omega$. The norm is defined as\[ \Vert u\Vert_{p,\Omega}=\left(\int_{\Omega}| u(x)|^p dx\right)^{\frac{1}{p}}\]
Also, \[\Vert u\Vert_{\infty,\Omega}= \operatorname*{ess~sup}\lbrace |u(x)|:x\in\Omega\rbrace.\]

If $p\geq 1$, then $W^2_p(\Omega)$ is the Sobolev space of functions $u:\Omega\rightarrow \mathbb{R}$ with generalized derivatives, $\partial_x^s u$ (in the sense of distributions) $|s|\leq 2$ belonging to $L_p(\Omega)$.  Here $s=(s_1,s_2,...,s_n)$, $|s|=s_1+s_2+...+s_n$, $|s|\leq2$, and $\partial_x^{s}=\partial_1^{s_1}\partial_2^{s_2}...\partial_n^{s_n}$ where $\partial_i=\frac{\partial}{\partial x_i}$. The norm in this space is \[\Vert u\Vert_{p,\Omega}^{(2)}=\sum_{|s|=0}^{2}\Vert \partial_x^s u\Vert_{p,\Omega} \]

Similarly, $W^{(2,1)}_p(\Omega_T)$ is the Sobolev space of functions $u:\Omega_T\rightarrow \mathbb{R}$ with generalized derivatives, $\partial_x^s\partial_t^r u$ (in the sense of distributions) where $2r+|s|\leq 2$ and each derivative belonging to $L_p(\Omega_T)$. The norm in this space is \[\Vert u\Vert_{p,\Omega_T}^{(2,1)}=\sum_{2r+|s|=0}^{2}\Vert \partial_x^s\partial_t^r u\Vert_{p,\Omega_T}. \]

In addition to the spaces above, we also make reference to the well known spaces of continuous functions and continuously differentiable functions. For a rigorous treatment of these spaces, and the associated spaces on $M$ and $M_T$, we refer the reader to Chapter 2 of \cite{RefWorks:65}.\\

\begin{definition}\label{def}
 A function $u$ is said to be a solution of $(\ref{primary})$ if and only if \[ u\in C(\overline\Omega\times[0,T),\mathbb{R}^m)\cap C^{1,0}(\overline\Omega\times(0,T),\mathbb{R}^m)\cap C^{2,1}(\Omega\times(0,T),\mathbb{R}^m)\] such that $u$ satisfies $(\ref{primary})$. If $T=\infty$ then the solution is said to be a global solution.
\end{definition}\\

\noindent We start by stating a local well posedness result that was proved in $\cite{RefWorks:1}$.\\

\begin{theorem}\label{local}
Suppose $(V_N)$, $(V_{F})$, and $(V_{QP})$ holds. Then there exists $T_{\max}>0$ such that $\left (\ref{primary}\right)$ has a unique, maximal, component-wise nonegative solution $u$ with $T=T_{\max}$.  Moreover, if $T_{\max}<\infty$ then 
\[\displaystyle \limsup_{t \to T^-_{\max}}\Vert u(\cdot,t)\Vert_{\infty,\Omega}=\infty. \]
\end{theorem}\\

\noindent According to Theorem \ref{local}, global existence is guaranteed provided we can obtain a priori sup norm bounds for each component of our solution. This leads us immediately to ask whether the results in \cite{RefWorks:5} can be extended to this setting. We give a partial response in the result below.\\

\begin{theorem}\label{martinthm}
Suppose $m=2$ and $(V_N)$, $(V_{F})$, $(V_{QP})$, $(V_{L1})$ and $(V_{Poly})$ hold, and let $T_{\max}>0$ be given in Theorem \ref{local}. If there exists a nondecreasing function $h\in C(\mathbb{R}_+,\mathbb{R}_+)$ such that $\|u_i(\cdot,t)\|_{\infty,\Omega}\le h(t)$ for all $0\le t<T_{\max}$, for either $i=1$ or $i=2$, and there exists $K>0$ so that whenever $a\ge K$ there exists $L_a\ge 0$ so that 
\begin{align}\label{nearlykoua}
aG_1(z)+G_2(z)\le L_a(z_1+z_2+1),\quad\text{for all}\quad z\in\mathbb{R}_+^2,
\end{align}
then  (\ref{primary}) has a unique component-wise nonegative global solution.  
\end{theorem}\\

\noindent A corollary of the proof of Theorem \ref{martinthm} is that if the assumption (\ref{nearlykoua}) is omitted, then finite time blow up can only occur near the boundary. \\

\begin{corollary}\label{cor1}
Suppose $m=2$ and $(V_N)$, $(V_{F})$, $(V_{QP})$, $(V_{L1})$ and $(V_{Poly})$ hold, and let $T_{\max}>0$ be given in Theorem \ref{local}. If there exists a nondecreasing function $h\in C(\mathbb{R}_+,\mathbb{R}_+)$ such that $\|u_i(\cdot,t)\|_{\infty,\Omega}\le h(t)$ for all $0\le t<T_{\max}$, for either $i=1$ or $i=2$, then for every open subset $W\subset\Omega$ such that $\overline{W}\subset\Omega$, there exists there exists a nondecreasing function $h_W\in C(\mathbb{R}_+,\mathbb{R}_+)$ such that $\|u_i(\cdot,t)\|_{\infty,W}\le h_W(t)$ for all $0\le t<T_{\max}$, for both $i=1$ and $i=2$.
\end{corollary}\\

Note that (\ref{nearlykoua}) is a {\em portion} of $(V_{L})$ in the case $m=2$. It turns out that the full extend of $(V_L)$ is a useful tool for obtaining a priori estimates and proving global existence when $m\ge 2$.\\

\begin{theorem}\label{global}
Suppose $(V_N)$, $(V_{F})$, $(V_{QP})$, $(V_{L})$ and $(V_{Poly})$ hold. Then  $(\ref{primary})$ has a unique component-wise nonegative global solution.
\end{theorem}\\

\noindent This global existence can also give rise to a uniform bound, provided an $L_1(\Omega)$ bound can be obtained for every component of the solution.\\

\begin{theorem}\label{uniform}
Suppose $(V_N)$, $(V_{F})$, $(V_{QP})$, $(V_{L})$ and $(V_{Poly})$ hold and $\Vert u\Vert_{1, \Omega\times (\tau, \tau+1)} $  is bounded independent of $\tau>0$. Then $(\ref{primary})$ has a unique, componentwise nonnegative global solution that is uniformly bounded in the sup norm. 
\end{theorem}\\

\noindent Finally, the condition $(V_{L1})$ can be used to obtain an $L_1(\Omega)$ bound when $L_1=0$. As a result, we have the following corollary.\\

\begin{corollary}\label{cor2}
If the hypotheses of Theorem \ref{global} are satisfied, and additionally $(V_{L1})$ is satisfied with $L_1=0$, then $\Vert u(\cdot, \tau)\Vert_{1,\Omega}$ is bounded independent of $\tau>0$, and the conclusion of Theorem \ref{uniform} is true. 
\end{corollary}\\

We give some estimates for solutions of linear equations in the next section, and provide the proofs of our main results in the sections that follow.

\section {Estimates for Solutions of Linear Equations}
The estimates below will play a fundamental role in the work that follows. Let $d,T>0$, $N_1,N_2\in\mathbb{R}$, and consider the system
\begin{align}\label{neumann}
\varphi_t&=d\Delta \varphi+N_1\varphi+\theta &x\in\Omega,0<t<T \nonumber\\
d\frac{\partial\varphi}{\partial\eta}&=N_2\varphi+\gamma, & x\in M, 0<t<T, \\
\varphi&=\varphi_0 &x\in\Omega,t=0 \nonumber
\end{align}

\noindent The result below is a consequence of the proof of Theorem 9.1 in \cite{RefWorks:65}, and the comment following the proof on page 351.\\

\begin{lemma}\label{neu}
Let $p>1$. Suppose $\theta\in L_p(\Omega\times(0,T))$, $\varphi_0\in W_p^{(2-\frac{2}{p})}(\Omega)$, and $\gamma\in W_p^{(1-\frac{1}{p},\frac{1}{2}-\frac{1}{2p})}(M\times(0,T))$ with $p\ne 3$. In addition, when $p>3$ assume
\[ d\frac{\partial\varphi_0}{\partial\eta}=N_2\varphi_0+\gamma  \ \text{on} \ M\times\lbrace 0\rbrace.\]
Then $(\ref{neumann})$  has a unique solution $\varphi\in W_p^{2,1}(\Omega\times(0,T))$ and there exists $C$ dependent upon $\Omega$, $p$, $T$, $N_1$, $N_2$ and $d$, and independent of $\theta$, $\varphi_0$ and $\gamma$, such that 
\[ \Vert \varphi\Vert_{p,(\Omega\times(0,T))}^{(2,1)}\leq C\left(\Vert \theta\Vert_{p,(\Omega\times(0,T))}+\Vert\varphi_0\Vert_{p,\Omega}^{(2-\frac{2}{p})}+\Vert \gamma\Vert_{p,(\partial\Omega\times(0,T))}^{(1-\frac{1}{p}, \frac{1}{2}-\frac{1}{2p})}\right)\]
\end{lemma}\\

The next result is given in section 5, Theorem 3.6, of \cite{RefWorks:1}.\\

\begin{lemma}\label{holder} Suppose $p>n+1$, and $\theta\in L_p(\Omega\times(0,T))$, $\gamma\in L_p(M\times(0,T))$, $N_1=N_2=0$ and $\varphi_0\in W^{2}_p(\Omega)$ such that 
\[ d\frac{\partial\varphi_0}{\partial\eta}=\gamma(x,0) \ on \ M.\]
Then there exists $C_{p,T}>0$ independent of $\theta$, $\gamma$ and $\varphi_0$, and the unique weak solution $\varphi\in V_2^{1,\frac{1}{2}}(\Omega_T)$ of $(\ref{neumann})$, such that if $0<\beta<1-\frac{n+1}{p}$ then  \[ \vert \varphi\vert_{\Omega_{T}}^{(\beta)}\leq C_{p,T}\left(\Vert\theta\Vert_{p,\Omega_T}+\Vert\gamma\Vert_{p,M_T}+\Vert\varphi_0\Vert_{p,\Omega}^{(2)}\right),\]
where $\vert \varphi\vert_{\Omega_{\hat T}}^{(\beta)}$ is the H\"older norm of $\varphi$ with exponent $\beta$.
\end{lemma}\\

We conclude this section with the following {\em seemingly} well known result, which plays an important role in proof of Theorems \ref{global} and \ref{uniform}. For lack of a good reference, we have included the proof.\\

\begin{lemma}\label{omega}
If $\gamma\geq 1$ and $\epsilon>0$, then there exists $M_{\epsilon, \gamma}>0$ such that 
\begin{align}\label{firstone} 
\Vert v\Vert_{2,\Omega}^2\leq \epsilon\Vert\nabla v\Vert_{2,\Omega}^2+M_{\epsilon,\gamma}\Vert v^{\frac{2}{\gamma}}\Vert_{1,\Omega}^{\gamma}
\end{align}
\begin{align}\label{secondone} 
\Vert v\Vert_{2,M}^2\leq \epsilon\Vert\nabla v\Vert_{2,\Omega}^2+M_{\epsilon,\gamma}\Vert v^{\frac{2}{\gamma}}\Vert_{1,\Omega}^{\gamma}
\end{align}
for all $v\in H^1(\Omega)$.
\end{lemma}
\begin{proof}
We start with (\ref{firstone}). Let $\gamma\geq 1$ and $\epsilon>0$. Suppose by way of contradiction that for every natural number $k$, there is a function $v_k\in H^{1}(\Omega)$ such that 
\[ \Vert v_k\Vert_{2,\Omega}^{2}\geq \epsilon \Vert \nabla v_k\Vert_{2,\Omega}^2+k\Vert v_k^{\frac{2}{\gamma}}\Vert^{\gamma}_{1,\Omega}\]for all $k$. From the homogenity of the inequality, we can assume \[ \Vert v_k\Vert_{2,\Omega}^2=1\] for all $k$. As a result, the sequence $\lbrace v_k\rbrace$ is bounded in $H^{1}(\Omega)$. In addition \[ \Vert v_k^{\frac{2}{\gamma}}\Vert_{1,\Omega}\rightarrow 0 \quad\text {as} \quad k\rightarrow \infty\] Now, since $H^{1}(\Omega)$ is compactly embedded in $L_2(\Omega)$, there is a subsequence $\lbrace v_{k_j}\rbrace$ of $\lbrace v_k\rbrace$ and a function $v\in L_2(\Omega)$ such that $\Vert v_{k_j} -v\Vert_{2,\Omega}\rightarrow 0$ as $j\rightarrow \infty$. However, from above, it is apparent that $\Vert v^{\frac{2}{\gamma}}\Vert_{1,\Omega}=0$, implying $v=0$ almost everywhere, which contradicts the fact that $\Vert v\Vert_{2,\Omega}=\lim_{j\rightarrow \infty}\Vert v_{k_j}\Vert_{2,\Omega}=1$. Therefore (\ref{firstone}) is true. Finally, (\ref{secondone}) follows from (\ref{firstone}) by applying equation (2.25) on page 49 in \cite{RefWorks:69}.
\end{proof}

\section{Proofs of Theorem \ref{martinthm} and Corollary \ref{cor1}} We begin with the proof of Theorem \ref{martinthm}. Assume $m=2$, and $(V_N)$, $(V_{F})$, $(V_{QP})$, $(V_{L1})$ and $(V_{Poly})$ hold. If $T_{\max}=\infty$, then there is nothing to do. So, assume $T=T_{\max}<\infty$. We can assume WLOG that we have $\|u_1(\cdot,t)\|_{\infty,\Omega}\le h(t)$ for all $0\le t<T_{\max}$, and that $b_1=b_2=1$ in $(V_{L1})$. Let $1<p<\infty$ and set $p'=\frac{p}{p-1}$. Suppose $\theta\in L_{p'}(\Omega_T)$ such that $\theta\ge 0$ and $\|\theta\|_{p',\Omega_T}=1$. Furthermore, let $L_2\geq \max \lbrace \frac{d_2L_1}{d_1},L_1\rbrace$ and suppose $\varphi$ solves 
\begin{align}\label{phieq}
\varphi_t+d_2\Delta\varphi&=-L_1\varphi-\theta&\text{on }\Omega_T,\nonumber\\
d_2\frac{\partial}{\partial \eta}\varphi&=L_2\varphi&\text{on }M_T,\\
\varphi&=0&\text{on }\Omega\times\left\{T\right\}.\nonumber
\end{align}
At first glance, (\ref{phieq}) may appear to be a backwards heat equation. However, the substitution $\tau=T-t$ immediately reveals that it is actually the forward heat equation. Moreover, $\varphi\ge 0$ from the same argument that is used to prove Theorem \ref{local}. In addition, from Lemma \ref{neu}, there is a constant $C>0$ dependent on $p$, $d_1$, $d_2$, $\Omega$, $L_1$ and $L_2$, and independent of $\theta$ such that 
$\|\varphi\|_{p',\Omega_T}^{(2,1)}\le C.$
Now we use integration by parts and $(V_{L1})$ to obtain
\begin{align}\label{duality1}
\int_0^T\int_\Omega(u_1+u_2)\theta dxdt&= \int_0^T\int_\Omega(u_1+u_2)(-\varphi_t-d_2\Delta\varphi-L_1\varphi)dxdt\nonumber\\
&=\int_\Omega(w_1+w_2)\varphi(x,0)dx+\int_0^T\int_\Omega\varphi((u_1)_t+(u_2)_t)dxdt\nonumber\\
&-\int_0^T\int_\Omega(u_1+u_2)(d_2\Delta\varphi-L_1\varphi)dxdt\nonumber\\
&\le\int_\Omega(w_1+w_2)\varphi(x,0)dx+\int_0^T\int_\Omega\varphi(d_1\Delta u_1+d_2\Delta u_2)dxdt\nonumber\\
&-\int_0^T\int_\Omega(u_1+u_2)d_2\Delta\varphi dxdt+\int_0^T\int_\Omega L_1\varphi dxdt,
\end{align}
from $(V_{L1})$. Also, note that integration by parts and $(V_{L1})$ imply
\begin{align}\label{duality1.1}
\int_0^T\int_\Omega\varphi(d_1\Delta u_1&+d_2\Delta u_2)dxdt=\int_0^T\int_M\varphi(G_1(u)+G_2(u))d\sigma dt\nonumber\\
&-\int_0^T\int_\Omega(d_1u_1+d_2u_2)\frac{L_2}{d_2}\varphi d\sigma dt+\int_0^T(d_1u_1+d_2u_2)\Delta\varphi dxdt\nonumber\\
&\le\int_0^T\int_M\varphi L_1(u_1+u_2+1)d\sigma dt-\int_0^T\int_\Omega(d_1u_1+d_2u_2)\frac{L_2}{d_2}\varphi d\sigma dt\nonumber\\
&+\int_0^T(d_1u_1+d_2u_2)\Delta\varphi dxdt\nonumber\\
&\le\int_0^T\int_M\varphi \left[\left(L_1-\frac{d_1}{d_2}L_2\right)u_1+\left(L_1-L_2\right)u_2+L_1\right] d\sigma dt\nonumber\\
&+\int_0^T(d_1u_1+d_2u_2)\Delta\varphi dxdt\nonumber\\
&\le \int_0^T\int_M L_1\varphi d\sigma dt+\int_0^T(d_1u_1+d_2u_2)\Delta\varphi dxdt.
\end{align}
From the assumption on $L_2$. Therefore, if we combine (\ref{duality1}) and (\ref{duality1.1}), we have
\begin{align}\label{duality1.2}
\int_0^T\int_\Omega(u_1+u_2)\theta dxdt &\le \int_\Omega(w_1+w_2)\varphi(x,0)dx+\int_0^T\int_\Omega(d_1-d_2)u_1\Delta \varphi dxdt\nonumber\\
&+\int_0^T\int_\Omega L_1\varphi dxdt+\int_0^T\int_M L_1 \varphi dxd\sigma
\end{align}
Recall that $\|u_1(\cdot,t)\|_{\infty,\Omega}\le h(t)$, and $\| \varphi\|_{p',\Omega_T}^{(2,1)}\le C$. Also, integrating (\ref{phieq}) reveals that $\|\varphi(\cdot,0)\|_{1,\Omega}$ can be bounded independent of $\theta$, by using the norm bound on $\varphi$, and $\|\theta\|_{p',\Omega_T}=1$. In addition, trace embedding implies $\|\varphi\|_{1,M_T}$ can be bounded in terms of $\|\varphi\|_{p',\Omega_T}^{(2,1)}$, which can be bounded independent of $\theta$, for the same reason as above. Therefore, by applying duality to (\ref{duality1.2}), we see that $\|u_2\|_{p,\Omega_T}$ is bounded in terms of $p$, $h(T)$, $L_1$, $d_1$, $d_2$ and $C$. Also, since $1<p<\infty$ is arbitrary, we have this estimate for every $1<p<\infty$. Note that the sup norm bound on $u_1$, the $L_p(\Omega_T)$ bounds on $u_2$ for all $1<p<\infty$, and $(V_{Poly})$, imply we have $L_q(\Omega_T)$ bounds on $F_1(u)$ and $F_2(u)$ for all $1<q<\infty$. 

Now,we use the bounds above and assumption (\ref{nearlykoua}) to show $\|u_2\|_{p,M_T}$ for all $1<p<\infty$. To this end, suppose $p\in \mathbb{N}$ such that $p\ge 2$, and let $\theta>\max\left\{K,\frac{d_1+d_2}{2\sqrt{d_1d_2}}\right\}$. We will see the reason for this choice below. To this end, we employ a modification of an argument given in \cite{RefWorks:8} for the case $m=2$. To simplify notation, we define $u=(u_1,u_2)$, and if $a,b\ge 0$ then $u^{(a,b)}=u_1^au_2^b$. 

Define
$$L(t)=\int_\Omega \sum_{\beta=0}^p \frac{p!}{\beta!(p-\beta)!}\theta^{\beta^2}u^{(\beta,p-\beta)}dx.$$ 
Then
\begin{align}\label{Lprime1}
L'(t)&=\int_\Omega \sum_{\beta=0}^p \frac{p!}{\beta!(p-\beta)!}\theta^{\beta^2}\left(\beta u^{(\beta-1,p-\beta)}(u_1)_t+(p-\beta)u^{(\beta,p-\beta-1)}(u_2)_t\right)dx\nonumber\\
&=\int_\Omega \left(pu_2^{p-1}(u_2)_t+p\theta^{p^2}u_1^{p-1}(u_1)_t\right)dx+X_1+X_2,
\end{align}
where
\begin{align}\label{Lprime2.1}
X_1&=\int_\Omega \sum_{\beta=1}^{p-1} \frac{p!}{(\beta-1)!(p-\beta)!}\theta^{\beta^2} u^{(\beta-1,p-\beta)}(u_1)_t dx\nonumber\\
&=\int_\Omega p\theta u_2^{p-1}(u_1)_t dx+\int_\Omega \sum_{\beta=2}^{p-1} \frac{p!}{(\beta-1)!(p-\beta)!}\theta^{\beta^2} u^{(\beta-1,p-\beta)}(u_1)_t dx\nonumber\\
&=\int_\Omega p\theta u_2^{p-1}(u_1)_t dx+\int_\Omega \sum_{\beta=1}^{p-2} \frac{p!}{\beta!(p-\beta-1)!}\theta^{(\beta+1)^2} u^{(\beta,p-\beta-1)}(u_1)_t dx\nonumber\\
&=\int_\Omega p\theta u_2^{p-1}(u_1)_t dx+\int_\Omega \sum_{\beta=1}^{p-2} \frac{p!}{\beta!(p-\beta-1)!}\theta^{\beta^2} u^{(\beta,p-\beta-1)}\theta^{2\beta+1}(u_1)_t dx
\end{align}
and
\begin{align}\label{Lprime2.2}
X_2&=\int_\Omega \sum_{\beta=1}^{p-1} \frac{p!}{\beta!(p-\beta-1)!}\theta^{\beta^2} u^{(\beta,p-\beta-1)}(u_2)_t dx\nonumber\\
&=\int_\Omega p\theta^{(p-1)^2}u_1^{p-1}(u_2)_t dx+\int_\Omega \sum_{\beta=1}^{p-2} \frac{p!}{\beta!(p-\beta-1)!}\theta^{\beta^2} u^{(\beta,p-\beta-1)}(u_2)_t dx.
\end{align}
Combining (\ref{Lprime1})-(\ref{Lprime2.2}) gives
\begin{align}\label{Lprime4}
L'(t)&=\int_\Omega\sum_{\beta=0}^{p-1}\frac{p!}{\beta!(p-1-\beta)!}\theta^{\beta^2}u^{(\beta,p-1-\beta)}\left(\theta^{2\beta+1}(u_1)_t+(u_2)_t\right)dx\nonumber\\
&=I+II,
\end{align}
where
\begin{align}\label{Lprime5}
I=\int_\Omega\sum_{\beta=0}^{p-1}\frac{p!}{\beta!(p-1-\beta)!}\theta^{\beta^2}u^{(\beta,p-1-\beta)}\left(\theta^{2\beta+1}F_1(u)+F_2(u)\right)dx
\end{align}
and
\begin{align}\label{Lprime6}
II=\int_\Omega\sum_{\beta=0}^{p-1}\frac{p!}{\beta!(p-1-\beta)!}\theta^{\beta^2}u^{(\beta,p-1-\beta)}\left(\theta^{2\beta+1}d_1\Delta u_1+d_2\Delta u_2\right)dx.
\end{align}
Note that $\int_0^T Idx$ is bounded because $(V_{Poly})$ holds, and as we have shown above, $\|u_i\|_{q,\Omega_T}$ is bounded for $i=1,2$ for all $1<q<\infty$. 

Now, consider II. Similar to the calculations for $L'(t)$, we can show
\begin{align}\label{Lprime7}
II=&-\int_\Omega\sum_{\beta=0}^{p-2}\frac{p!}{\beta!(p-2-\beta)!}\theta^{\beta^2}u^{(\beta,p-2-\beta)}\sum_{k=1}^n\sum_{i,j=1}^2 b_{i,j}\frac{\partial u_i}{\partial x_k}\frac{\partial u_j}{\partial x_k}\nonumber\\
&+\int_M\sum_{\beta=0}^{p-1}\frac{p!}{\beta!(p-1-\beta)!}\theta^{\beta^2}u^{(\beta,p-1-\beta)}\left(\theta^{2\beta+1}G_1(u)+G_2(u)\right)d\sigma,
\end{align} 
where 
$$\left(b_{i,j}\right)=\begin{pmatrix}d_1\theta^{4\beta+4}&\frac{d_1+d_2}{2}\theta^{2\beta+1}\\\frac{d_1+d_2}{2}\theta^{2\beta+1}&d_2\end{pmatrix}.$$
From the choice of $\theta$, this matrix is positive definite, so there exists $\alpha_{\theta,p}>0$ such that
\begin{align}\label{Lprime8}
L'(t)+&\alpha_{\theta,p}\int_\Omega \left(|\nabla (u_1)^{p/2}|^2+|\nabla (u_2)^{p/2}|^2\right)dx\le I \nonumber\\
&+ \int_M\sum_{\beta=0}^{p-1}\frac{p!}{\beta!(p-1-\beta)!}\theta^{\beta^2}u^{(\beta,p-1-\beta)}L_{\theta^{2\beta+1}}\left(u_1+u_2+1\right)d\sigma\nonumber\\
&\le I+N_{p,\theta,M}\left[\int_M\left(u_1^p+u_2^p\right)d\sigma+1\right]
\end{align} 
from (\ref{nearlykoua}), for some $N_{p,\theta,M}>0$. So, if we apply Lemma \ref{omega}, we can see there exists $\tilde{N}_{p,\theta,M}>0$ such that
\begin{align}\label{Lprime9}
L'(t)+N_{p,\theta,M}\int_M(u_1^p+u_2^p)d\sigma\le I+\tilde{N}_{p,\theta,M}\left(\int_{\Omega}(u_1+u_2)\ dx \right)^p+N_{p,\theta,M}.
\end{align}
Finally, if we integrate over time, we find that $\|u_2\|_{p,M_T}$ is bounded in terms of $p$, $M$, $\Omega$, $\theta$, $h(T)$, $w_1$, $w_2$ and $\|u_2\|_{p,\Omega_T}$. Since this holds for every natural number $p\ge 2$, we can use the assumption $(V_{Poly})$ and the bounds above, along with Lemma \ref{holder} to conclude that $\|u\|_{\infty,\Omega_T}<\infty$. From Theorem \ref{local}, this contradicts our assumption that $T_{\max}<\infty$. Therefore, $T_{\max}=\infty$, and Theorem \ref{martinthm} is proved.

Now, let's prove Corollary \ref{cor1}. Note that from the first portion of the proof above, we have $L_q(\Omega_T)$ bounds on $F_1(u)$ and $F_2(u)$ for all $1<q<\infty$. Let $W$ be an open subset of $\Omega$ such that $\overline{W}\subset\Omega$, and choose an open subset $\tilde{W}$ of $\Omega$ with smooth boundary, such that $\overline{W}\subset\tilde{W}$. Then, from the proof of Theorem 9.1 in \cite{RefWorks:65}, we are assured that if $1<q<\infty$ then there exists $C>0$ dependent on $q$, $d_i$ and the distance from $\partial W$ to $M$, such that
$$\|u_i\|_{q,\tilde{W}\times(0,t)}^{(2,1)}\le C\left(\|F_i(u)\|_{q,\Omega_t}+\|w_i\|_{C^2(\overline\Omega)} \right).$$
If we choose $q$ sufficiently large, then we get the result. 

\section{Proofs of Theorems \ref{global} and \ref{uniform}, and Corollary \ref{cor2}}
In order to derive $L_p$ estimates of $u$ on $\Omega$ and $M$, we create a functional defined in \cite{RefWorks:8}. To this end, let $A_{ij}=\frac{d_i+d_j}{2\sqrt{d_id_j}}$ for all $i,j=1, ....,m$, and, as in \cite{RefWorks:8}, for $i=1,...,m-1$, let $\theta_i>0$, such that \[ K_l^l>0\quad \quad \text {for}\quad l=2, ..., m,\] where \[ K_l^r=K_{r-1}^{r-1}\cdot K_l^{r-1}-[H_l^{r-1}]^2, \quad r=3, ..., l,\]
\[ H_l^r= \det_{1\leq i,j\leq l}\left( ({d_{i,j}})_{\substack{{i\ne l,...,r+1}\\{j\ne l-1,...,r}}}\right)\cdot \prod_{k=1}^{k=r-2}(\det \ [k])^{2^{(r-k-2)}}, \quad r=3, ...,l-1,\]
\[ K_l^2=\underbrace{d_1d_l \prod_{k=1}^{l-1}{\theta_k}^{2(p_k+1)^2}\cdot \prod_{k=l}^{m-1}{\theta_k}^{2(p_k+2)^2}}_{\text{positive values}} \cdot \left( \prod_{k=1}^{l-1}{\theta_k}^2-A_{1l}^2\right) \] and \[ H_l^2=\underbrace{d_1\sqrt{d_2d_l}{\theta_1}^{2(p_1+1)^2}\prod_{k=2}^{l-1}{\theta_k}^{(p_k+2)^2+(p_k+1)^2}\cdot \prod_{k=l}^{m-1}{\theta_k}^{2(p_k+2)^2}}_{\text{positive values}}\cdot \left({\theta_1}^2A_{2l}-A_{12}A_{1l}\right).\]
Here, $\det_{1\leq i,j\leq l}\left( ({d_{i,j}})_{\substack{{i\ne l,...,r+1}\\{j\ne l-1,...,r}}}\right)$ denotes the determinant of $r$ square symmetric matrix obtained from $(d_{i,j})_{1\leq i,j\leq m}$ by removing the $(r+1)\text{th}, (r+2)\text{th} ,....,l\text{th}$ rows and the $r\text{th}, (r+1)\text{th}, ...., (l-1)\text{th}$ columns, and $\det \ [1]$, ... , $\det \ [m]$ are the minors of the matrix $(a_{l,k})_{1\le l,k\le m}$.  The elements of the matrix $(d_{i,j})$ are \[ d_{ij}=\frac{d_i+d_j}{2}\theta_1^{(p_1)^2}\dots  \theta_{(i-1)}^{p_{i-1}^2} \theta_i^{(p_i+1)^2}\dots \theta_{j-1}^{(p_{j-1}+1)^2}\theta_j^{(p_j+2)^2}\dots \theta_{(m-1)}^{(p_{m-1}+2)^2} \]
The following lemma is given in $\cite{RefWorks:8}$.\\
\begin{lemma}\label{diff}
Let $H_{p_m}$ be the homogeneous polynomial such that \[H_{p_m}(u(x,t))= \sum_{p_{m-1}}^{p_m}\cdot\cdot\cdot\sum_{{p_1}=0}^{p_2}C_{p_m}^{p_{m-1}}\cdot\cdot\cdot C_{p_2}^{p_1}\theta_1^{{p_1}^2}\cdot\cdot\cdot\theta_{(m-1)}^{p^2_{(m-1)}}{u_1}^{p_1}{u_2}^{p_2-p_1}\cdot\cdot\cdot {u_m}^{p_m-p_{m-1}}\] with $p_m\geq 2$ being a positive integer, $C_{p_j}^{p_i}=\frac{p_j!}{p_i!(p_j-p_i)!}$, and $\theta_i\geq 0$ for all $i$. Then \[ \partial_{u_i}H_{p_m}= p_m \sum_{p_{m-1}=0}^{p_{m}-1}\cdot\cdot\cdot \sum_{p_1=0}^{p_2}C_{p_{m}-1}^{p_{m-1}}\cdot\cdot\cdot C_{p_2}^{p_1}\theta_1^{{p_1}^2}\cdot\cdot\cdot \theta_{i-1}^{p^2_{(i-1)}}\theta_i^{(p_{i}+1)^2} \cdot\cdot\cdot \theta_{(m-1)}^{(p_{(m-1)}+1)^2} \times u_1^{p_1}u_2^{p_2-p_1}\cdot\cdot\cdot u_m^{(p_m-1)-p_{m-1}}  \] for all $i=2, ..., m-1$.
\end{lemma}\\

We first establish an $L_1$ estimate for solutions to (\ref{primary}).\\

\begin{lemma}
Suppose that $(V_N)$, $(V_F)$, $(V_{QP})$ and $(V_{L1})$ are satisfied, and $u$ is the unique, componentwise nonnegative, maximal solution to (\ref{primary}). Then for all $0<t<T_{max}$, \[ \Vert u(\cdot, t)\Vert_{1,\Omega}\leq \alpha(t)\]for some nondecreasing continuous function $\alpha$ dependent on $L_1$ and $b_1,...,b_m$ in $(V_{L1})$. In addition, if $L_1=0$ then $\Vert u(\cdot, t)\Vert_{1,\Omega}$ is bounded independent of $t\geq 0$.
\end{lemma}
\begin{proof}
WLOG assume $b_i=1$ for all $i=1,...,m$. Integrating the $u_j$ equation over $\Omega$, we get
\begin{align}\label{interior}
\frac{d}{dt}\int_{\Omega} \sum_{j=1}^mu_j&=\sum_{j=1}^m\int_{\Omega} d_j\Delta u_j+\int_{\Omega}\sum_{j=1}^mF_j(u)\nonumber\\
 & \leq \int_{\Omega} \sum_{j=1}^mF_j(u)+ \int_M \sum_{j=1}^mG_j(u)\nonumber\\
 &\leq \int_{\Omega}L_1\left(\sum_{j=1}^m{u_j+1}\right)+\int_M L_1\left(\sum_{j=1}^m{u_j+1}\right).
\end{align}
Note that if $L_1=0$, then (\ref{interior}) implies $\Vert u(\cdot, t)\Vert_{1,\Omega}$ is a priori bounded independent of $t\geq 0$. Now, suppose $0<T<T_{max}$, $L_1>0$, and let $d>0$. Consider the system 
\begin{align}
\varphi_t &=-d\Delta \varphi-L_1\varphi &  (x,t)\in\Omega\times (0, T) \nonumber\\
d\frac{\partial\varphi}{\partial\eta}&=L_1\varphi+1& (x,t)\in M\times (0, T) \nonumber\\
\varphi&=\varphi_{T}& x\in\Omega, \ t=T,
\end{align}
where $\varphi_{T}\in C^{2+\gamma}(\overline\Omega)$ for some $\gamma>0$, is strictly positive and satisfies the compatibility condition
\[ d \frac{\partial\varphi_{T}}{\partial\eta}=L_1\ \text {on} \ M\times \lbrace T\rbrace.\]
From Theorem 5.3 in chapter 4 of $\cite{RefWorks:65}$, $\varphi\in C^{2+\gamma, 1+\frac{\gamma}{2}}(\overline\Omega\times [0,T])$, and therefore $\varphi\in C^{2+\gamma, 1+\frac{\gamma}{2}}(M\times [0,T])$ . Also, similar to our comments in the previous section, $\varphi\ge 0$.
Now, consider
\begin{align}
0&=\int_{0}^{T}\int_{\Omega} u_j(-\varphi_t-d\Delta\varphi-L_1\varphi)\nonumber\\
&= \int_{0}^{T}\int_{\Omega} \varphi (u_{jt}-d_j \Delta u_j)-L_1\int_0^T\int_\Omega u_j\varphi-\int_{0}^{T}\int_M  u_j d\frac{\partial\varphi}{\partial\eta}+(d_j-d)\int_0^T\int_\Omega u_i\Delta\varphi\nonumber\\
&+\int_{0}^{T}\int_M \varphi d_j\frac{\partial u_j}{\partial\eta}+\int_{\Omega}  u_j(x,0)\varphi(x,0)-\int_{\Omega} u_j(x,T)\varphi(\cdot, T)\nonumber\\
&=\int_0^T \int_\Omega\varphi F_j(u)-L_1\int_0^T\int_\Omega u_j\varphi-\int_{0}^{T}\int_M  u_j(L_1\varphi+1)+(d_j-d)\int_0^T\int_\Omega u_i\Delta\varphi\nonumber\\
&+\int_{0}^{T}\int_M \varphi G_j(u)+\int_{\Omega}  u_j(x,0)\varphi(x,0)-\int_{\Omega} u_j(x,T)\varphi(\cdot, T).
\end{align}
Summing these equations, and and making use of $(V_{L1})$, gives
\begin{align}\label{eqn2}
\int_0^T\int_M \sum_{j=1}^m u_j\le \int_0^T\int_\Omega L_1\varphi&+\int_0^T\int_M L_1\varphi+\sum_{j=1}^m(d_j-d)\int_0^T\int_\Omega u_i\Delta\varphi\\
&+\int_{\Omega} \sum_{j=1}^m w_j(x)\varphi(x,0)-\int_M \sum_{j=1}^m u_j(x,T)\varphi_T(x).\nonumber
\end{align}
Now, recall that $\varphi_T$ is strictly positive. Let $0<\delta\le\varphi(x)$ for all $x\in\Omega$. Then (\ref{eqn2}) implies
\begin{align}\label{eqn2.1}
\delta\int_M \sum_{j=1}^m u_j(x,T)+\int_0^T\int_M \sum_{j=1}^m u_j\le \int_0^T\int_\Omega L_1\varphi&+\int_0^T\int_M L_1\varphi+\sum_{j=1}^m(d_j-d)\int_0^T\int_\Omega u_i\Delta\varphi\nonumber\\
&+\int_{\Omega} \sum_{j=1}^m w_j(x)\varphi(x,0).
\end{align}
Then, there exist constants $C_1,C_2>0$, depending on $L_1$, $d$, $\varphi_T$, $w_1,...,w_m$, $d_1,...,d_m$, and at most exponentially on $T$, such that 
 \begin{align}\label{eqn2.2}
\delta\int_M \sum_{j=1}^m u_j(x,T)+\int_0^T\int_M \sum_{j=1}^m u_j\le C_1+C_2\int_0^T\int_\Omega\sum_{j=1}^mu_j.
\end{align}
Now, return to (\ref{interior}), and integrate both sides in $t$ to obtain 
\begin{align}\label{eqn2.3}
\int_\Omega\sum_{j=1}^m u_j(x,t)dx\le L_1\left(\int_0^t\int_\Omega \sum_{j=1}^m u_j+\int_0^t\int_M\sum_{j=1}^m u_j+t|M|+t|\Omega|\right)+\int_\Omega\sum_{j=1}^mw_j(x).
\end{align}
The second term on the right hand side of (\ref{eqn2.3}) can be bounded above by $L_1$ times the right hand side of (\ref{eqn2.2}). Using this estimate, and Gronwall's inequality, we can obtain a bound for $\int_0^T\int_\Omega\sum_{j=1}^mu_j$ that depends on $T$. Placing this on the right hand side of (\ref{eqn2.2}) gives a bound for $\int_M \sum_{j=1}^m u_j(x,T)$ that depends on $T$. Applying this to the second integral on the right hand side of (\ref{interior}), and using Gronwall's inequality, gives the result.
\end{proof}
\begin{lemma} \label{lpestimate}
Suppose that $(V_N)$, $(V_F)$, $(V_{QP})$ and $(V_{L})$ are satisfied, and $u$ is the unique, componentwise nonnegative, maximal solution to (\ref{primary}).  If $1<p<\infty$ and $T=T_{max}<\infty$, then $\|u\|_{p,\Omega_T}$ and $\|u\|_{p,M_T}$ are bounded.
\end{lemma}
\begin{proof}
Note that $(V_L)$ implies $(V_{L1})$, and consequently, we can make use of our previous lemma. Consider the functional \[ L(t)=\int_{\Omega}H_{p_m}(u(x,t)) dx\] where $H_{p_m}(u(x,t))$ is given in Lemma $\ref{diff}$ with $p_m\geq 2$ is a positive integer. It is simple matter to prove that there are constant $\alpha_{p_m}$, $\beta_{p_m}>0$ depending on the $\theta_i$ so that 
\[ \alpha_{p_m}\left(\sum_{j=1}^{m}z_j\right)^{p_m}\leq H_{p_m}(z)\leq  \beta_{p_m}\left(\sum_{j=1}^{m}z_j\right)^{p_m}\]
for all $z\in \mathbb{R}^m_{+}.$
Now differentiating $L$ with respect to $t$ yields
\begin{align*}
L'(t) &=\int_{\Omega}\partial_t H_{p_m}(u)dx \nonumber\\
&= \int_{\Omega}\sum_{i=1}^m\partial_{u_i}H_{p_m}(u)\frac{\partial u_i}{\partial t} dx\nonumber\\
&=  \int_{\Omega}\sum_{i=1}^m\partial_{u_i}H_{p_m}(u)(d_i\Delta u_i+F_i) dx\nonumber\\
&=\int_{\Omega}\sum_{i=1}^m\partial_{u_i}H_{p_m}(u)d_i\Delta u_i dx+ \int_{\Omega}\sum_{i=1}^m\partial_{u_i}H_{p_m}(u)F_i (u) dx\nonumber\\
\end{align*}
Using Green's formula, we get 
\begin{align*}
L'(t)&=\int_{\Omega}\sum_{i=1}^m\partial_{u_i}H_{p_m}(u)d_i\Delta u_i dx+ \int_{\Omega}\sum_{i=1}^m\partial_{u_i}H_{p_m}(u) F_i (u) dx\nonumber\\
&= \int_{M}\sum_{i=1}^{m}\partial_{u_i} d_i H_{p_m}(u)\partial_{\eta}u_i ds-\int_{\Omega}\left[\left(\left(\frac{d_i+d_j}{2}\partial_{u_j u_i}H_{p_m}(u)\right)_{1\leq i,j\leq m}\right)V\right] \cdot V dx \\ \nonumber &\quad + \int_{\Omega}\sum_{i=1}^m\partial_{u_i}H_{p_m}(u)F_i(u)  dx,\nonumber
\end{align*}
for $p_1=0,...,p_2,p_2=0,...,p_3,...,p_{m-1}=0,...,p_m-2$ and $V=(\nabla u_1,\nabla u_2,...,\nabla u_m)^t$. So,
\begin{align}
L'(t) &+\int_{\Omega}\left[\left(\left(\frac{d_i+d_j}{2}\partial_{u_j u_i}H_{p_m}(u)\right)_{1\leq i,j\leq m}\right)V\right] \cdot V dx \\ \nonumber &= \int_{M}\sum_{i=1}^{m}\partial_{u_i} d_i H_{p_m}(u)\partial_{\eta}u_i ds+ \int_{\Omega}\sum_{i=1}^m\partial_{u_i}H_{p_m}(u)F_i(u)  dx\nonumber\\
&=\int_{M}\sum_{i=1}^{m} \partial_{u_i}H_{p_m}(u)G_i(u) ds+ \int_{\Omega}\sum_{i=1}^m\partial_{u_i}H_{p_m}(u)F_i(u)  dx\nonumber
\end{align}
From Lemma $\ref{diff}$, we know \[ \partial_{u_i}H_{p_m}(u)= p_m \sum_{p_{m-1}=0}^{p_{m}-1}\cdot\cdot\cdot \sum_{p_1=0}^{p_2}C_{p_{m}-1}^{p_{m-1}}\cdot\cdot\cdot C_{p_2}^{p_1}\theta_1^{{p_1}^2}\cdot\cdot\cdot \theta_{i-1}^{p^2_{(i-1)}}\theta_i^{(p_{i}+1)^2} \cdot\cdot\cdot \theta_{(m-1)}^{(p_{(m-1)}+1)^2} \times u_1^{p_1}u_2^{p_2-p_1}\cdot\cdot\cdot u_m^{(p_m-1)-p_{m-1}}  \]
As a result, \begin{align}
\int_{\Omega}\sum_{i=1}^m \partial_{u_i}H_{p_m} &(u) F_i(u)  dx \nonumber \\ \nonumber &=\int_{\Omega} \left[  p_m \sum_{p_{m-1}=0}^{p_{m-1}}\cdot\cdot\cdot \sum_{p_1=0}^{p_2}C_{p_{m}-1}^{p_{m-1}}\cdot\cdot\cdot C_{p_2}^{p_1} u_1^{p_1} u_2^{p_2-p_1}\cdot\cdot\cdot u_m^{p_m-1-p_{m-1}}\right]\nonumber\\ &\times\left(\prod_{i=1}^{m-1}\theta_i^{(p_i+1)^2}F_1(u)+\sum_{j=2}^{m-1}\prod_{k=1}^{j-1}\theta_k^{{p_k}^{2m-1}}\prod _{i=j}\theta_i^{(p_i+1)^2}F_j(u)+\prod_{i=1}^{m-1}\theta_i^{{p_i}^2} F_m(u)\right) dx\nonumber\\
&=\int_{\Omega} \left[  p_m \sum_{p_{m-1}=0}^{p_{m-1}}\cdot\cdot\cdot \sum_{p_1=0}^{p_2}C_{p_{m}-1}^{p_{m-1}}\cdot\cdot\cdot C_{p_2}^{p_1} u_1^{p_1} u_2^{p_2-p_1}\cdot\cdot\cdot u_m^{p_m-1-p_{m-1}}\right]\nonumber\\ &\times\left(\frac{\prod_{i=1}^{m-1}\theta_i^{(p_i+1)^2}}{\prod_{i=1}^{m-1}\theta_i^{{p_i}^2}}F_1(u)+\sum_{j=2}^{m-1}\frac{\prod_{k=1}^{j-1}\theta_k^{{p_k}^{2m-1}}\prod _{i=j}\theta_i^{(p_i+1)^2}}{\prod_{i=1}^{m-1}\theta_i^{{p_i}^2}}F_j(u)+F_m(u)\right)\prod_{i=1}^{m-1}\theta_i^{{p_i}^2} dx \nonumber\\
&=\int_{\Omega} \left[  p_m \sum_{p_{m-1}=0}^{p_{m-1}}\cdot\cdot\cdot \sum_{p_1=0}^{p_2}C_{p_{m}-1}^{p_{m-1}}\cdot\cdot\cdot C_{p_2}^{p_1} u_1^{p_1} u_2^{p_2-p_1}\cdot\cdot\cdot u_m^{p_m-1-p_{m-1}}\right]\nonumber\\ &\times\left(\prod_{i=1}^{m-1}\frac{\theta_i^{(p_i+1)^2}}{\theta_i^{{p_i}^2}}F_1(u)+\sum_{j=2}^{m-1} \prod_{i=j}^{m-1}\frac{\theta_i^{(p_i+1)^2}}{\theta_i^{{p_i}^2}}F_j(u)+F_m(u)\right)\prod_{i=1}^{m-1}\theta_i^{{p_i}^2} dx. \end{align}
Therefore,
\begin{align}
&\int_{\Omega}\sum_{i=1}^m\partial_{u_i}H_{p_m}(u)F_i(u)  dx 
\nonumber \\ \nonumber &\leq \hat C \int_{\Omega} \left[  p_m \sum_{p_{m-1}=0}^{p_{m-1}}\cdot\cdot\cdot \sum_{p_1=0}^{p_2}C_{p_{m}-1}^{p_{m-1}}\cdot\cdot\cdot C_{p_2}^{p_1} u_1^{p_1} u_2^{p_2-p_1}\cdot\cdot\cdot u_m^{p_m-1-p_{m-1}} \left (1+\sum_{i=1}^{m} u_i\right)\right]\nonumber\\
&\leq \hat L\int_{\Omega}\left(\sum_{i=1}^m u_i\right)^{p_m-1} \times  \left (1+\sum_{i=1}^{m} u_i\right)= \hat L\int_{\Omega} \left(1+\sum_{i=1}^m u_i\right)^{p_m} \nonumber\\&\leq L_{p_m} \int_{\Omega} \left(1+\sum_{i=1}^m u_i^{p_m}\right).
\end{align}
A similar calculation for $G_i(u)$ implies that for an appropriate choice of $c_{p_m}$ and $L_{p_m}>0$ we get
 \begin{align}\label{equation1}
L'(t)+c_{p_m}\int_{\Omega}\sum_{j=1}^{m}\vert \nabla u_j^{\frac{p_m}{2}}\vert^2 dx \leq L_{p_m}\left(\int_M\sum_{j=1}^m u_j^{p_m} d\sigma +\int_{\Omega}\sum_{j=1}^m u_j^{p_m} dx  +1\right).
\end{align}
From equation (2.25) on page 49 of $\cite{RefWorks:69}$, there exists constants $c_{p_m}$ and $M_1>0$ such that 
\begin{align}\label{equation2}
\Vert u\Vert_{L_{p_m} (M)}\leq \frac{c_{p_m}}{2}\Vert\nabla u\Vert_{L_2(\Omega)}+M_1\Vert u\Vert_{L_2(\Omega)}.
\end{align}
Now replacing $u$ by $u_j^{\frac{p_m}{2}}$, we get 
\begin{align}\label{equation3}
L_{p_m}\int_M \sum_{j=1}^m u_j^{p_m} d\sigma \leq \frac{c_{p_m}}{2}\int_{\Omega}\sum_{j=1}^{m}\vert \nabla u_j^{\frac{p_m}{2}}\vert dx+M_1\int_{\Omega}\sum_{j=1}^m u_j^{p_m} dx.
\end{align}
As a result, combining this with $(\ref{equation1})$, we have 
\begin{align}\label{equation4}
L'(t)+\frac{c_{p_m}}{2}\int_{\Omega}\sum_{j=1}^{m}\vert \nabla u_j^{\frac{p_m}{2}}\vert^2 dx \leq (L_{p_m}+M_1)\left(\int_{\Omega}\sum_{j=1}^m u_j^{p_m} dx \right)+L_{p_m}.
\end{align}
Now, we make use of to Lemma $\ref{omega}$ to conclude there is a constant $M_2>0$ such that 
\begin{align}\label{equation5}
(L_{p_m}+M_1+1)\int_{\Omega}\sum_{j=1}^m u_j^{p_m} dx\leq \frac{c_{p_m}}{2}\int_{\Omega}\sum_{j=1}^m \vert \nabla u_j^{\frac{p_m}{2}}\vert^2 dx +M_2\sum_{j=1}^m\left(\int_{\Omega} u_j dx\right)^{p_m}.
\end{align}
Combining  $(\ref{equation5})$ with $(\ref{equation4})$  and our $L_1$ estimates gives the existence of $M_3>0$ such that 
\begin{align}
L'(t) &=L_{p_m}+M_2\sum_{j=1}^m\left(\int_{\Omega} u_j dx\right)^{p_m}-\int_{\Omega}\sum_{j=1}^m u_j^{p_m} dx \nonumber\\
&\leq M_3-\alpha_p L(t)
\end{align}
for all $t\geq 0$. Consequently,
\[ L(t)\leq L(0) \exp (-\alpha_{p_m} t)+\frac{M_3}{\alpha_{p_m}} \]
for all $t\geq 0$. This inequality gives uniform $L_{p_m}(\Omega)$ estimates on $u$ for each $p_m>1$. Now return to (\ref{equation1}). This time, use the fact that there is a constant $M_4>0$ so that
\begin{align}\label{SD}
(L_{p_m}+1)\int_M\sum_{j=1}^m u_j^{p_m} d\sigma \leq c_{p_m}\int_{\Omega}\sum_{j=1}^m \vert\nabla u_j^{p_m/2}\vert^2 dx +M_4\int_{\Omega}\sum_{j=1}^m u_j^{p_m} dx
\end{align}
to obtain
\begin{align}\label{Manifold}
L'(t)+\int_M \sum_{j=1}^m u_j^{p_m} d\sigma \leq L_{p_m}+(L_{p_m}+M_4)\int_{\Omega}\sum_{j=1}^m u_j^{p_m} dx.
\end{align}
Integrating both the sides over the time interval $(0,T_{\max})$, and using the bounds derived above, gives an $L_{p_m}(M\times (0,T_{\max}))$ estimate on $u$ for each $p_m>1$.
\end{proof}

\hspace{1cm}\\
{\bf Proof of Theorem} \ref{global}: From Theorem $\ref{local}$, we already have a componentwise nonnegative, unique, maximal solution of $\ref{primary}$. If $T_{\max} = \infty$, then we are done. So, by way of contradiction assume $T_{\max} <\infty$. From Lemma $\ref{lpestimate}$ , we have $L_p$ estimates for our solution for all $p\geq 1$ on $\Omega\times (0,T_{\max})$ and $M\times (0,T_{\max})$. We know from $(V_{\text{Poly}})$  that the $F_i$ and $G_i$ are polynomially bounded above for each $i$. Then proceeding as in the proof of Theorem 3.3 in $\cite{RefWorks:1}$ with the bounds from Lemma \ref{lpestimate} we have $T_{\max}=\infty$.$\Box$\\

\begin{lemma} \label{lpuniformestimate}
Suppose that $(V_N)$, $(V_F)$, $(V_{QP})$ and $(V_{L})$ are satisfied, and $u$ is the unique, componentwise nonnegative, global solution to (\ref{primary}). If $\Vert u(\cdot,t)\Vert_{1, \Omega} $ is bounded independent of $t\geq 0$. Then $\|u\|_{p,\Omega\times(\tau,\tau+1)}$ and $\|u\|_{p,M\times(\tau,\tau+1)}$ are bounded, independent of $\tau\ge 0$, for each $p>1$.
\end{lemma}
\begin{proof}
The proof of Theorem $\ref{lpestimate}$ can be adopted to obtain this result by integrating over over $(\tau, \tau+1)$ for each $\tau\ge 0$. 
\end{proof}\\

{\bf Proof of Theorem} $\ref{uniform}$: Now, we convert these $L_p$ estimates obtained in Lemma $\ref{lpestimate}$ to sup norm estimates. For that purpose let $\tau\geq 0$ and define a cut off function $\psi\in C_0^{\infty}(\mathbb{R},[0,1])$ such that $\psi=1$ for all $t\geq \tau+1$ and $\psi(t)=0$ for all $t\leq \tau$. In addition, define $\hat u_i(x,t)=\psi(t)u_i(x,t)$. From construction $\hat u_i(x,t)=u_i(x,t)$ for all $(x,t)\in M\times (\tau+1,\tau+2)$ and $(x,t)\in\Omega\times(\tau+1,\tau+2)$ respectively. Also, the $\hat u_j$ satisfy the system 
\begin{align}\label{cutoff}
\frac{\partial \hat u_i}{\partial t}&=d_i \Delta \hat u_i+\psi'(t)u_j(x,t)+\psi(t) F_i(u)  \ & (x,t)\in \Omega\times(\tau,\tau+2)&\  \text{for} \  i=1, ..., m\nonumber\\
d_i\frac{\partial \hat u_i}{\partial\eta}&=\psi(t) G_i(u)  \ & (x,t)\in M\times(\tau,\tau+2)&\  \text{for}\  i=1, ..., m\\
u&=0  \quad & (x,0)\in \overline\Omega\times \tau  \nonumber
\end{align}
From $(V_{\text{Poly}})$, $F$ and $G$ are polynomially bounded above. Also, we have estimates for each of $\|\psi' u_j+\psi F_i(u)\|_{p,\Omega\times(\tau,\tau+2)}$ and $\|\psi G_i(u)\|_{p,M\times(\tau+\tau+2)}$ independent of $\tau\ge0$, for each $1<p<\infty$. Therefore, from Theorem $\ref{holder}$, if $p>n+1$, then $\hat u$ is sup norm bounded on $\Omega\times(\tau,\tau+2)$, independent of $\tau$. The result follows, since $\hat u(x,t)=u(x,t)$ when $\tau+1\le t\le \tau+2$.  
$\Box$

\section{Examples}
\subsection*{Example 1} We start with an example to illustrate the use of Theorem \ref{martinthm}. To this end, consider the system
\begin{align}\label{u1bounded}
u_{1_t}&=d_1 \Delta u_1+u_2^4(1-u_1)^3 \quad & x\in \Omega, t>0  \nonumber\\ 
u_{2_t}&=d_2\Delta u_2+u_2^4(u_1-1)^3 \quad & x\in \Omega, t>0 \nonumber\\
d_1\frac{\partial u_1}{\partial\eta}&=-u_1^2u_2^2 \quad & x\in M, t>0\\\nonumber
d_2 \frac{\partial u_2}{\partial\eta}&=u_1^2u_2^2 \quad & x\in M, t>0\\\nonumber
u_i(x,0)&=w_i(x) & x\in \overline\Omega, i=1,2 \nonumber
\end{align}
where $d_1, d_2>0$ and $w$ is sufficiently smooth and componentwise nonnegative. If we define 
$$F(u)=\begin{pmatrix}u_2^4(1-u_1)^3\\u_2^4(u_1-1)^3\end{pmatrix}\quad\text{and}\quad G(u)=\begin{pmatrix}-u_1^2u_2^2\\u_1^2u_2^2\end{pmatrix},$$
for all $u\in\mathbb{R}_+^2$, then we can easily see that $(V_N)$, $(V_{F})$, $(V_{QP})$ and $(V_{Poly})$ are satisfied. Also,   
$$F_1(u)+F_2(u)=0\quad\text{and}\quad G_1(u)+G_2(u)=0.$$
Furthmore, it is a simple matter to conclude that
$$\|u_1\|_\infty\le \max\left\{\|w_1\|_{\infty,\Omega},1\right\}$$
for all $u\in\mathbb{R}_+^2$. Consquently, we can apply Theorem \ref{martinthm} to conclude that (\ref{u1bounded}) has a unique, componentwise nonnegative, global solution. We remark that in this case, we can obtain a bound for $\|u_2(\cdot,t)\|_{1,\Omega}$ that is independent of $t\ge 0$ (by adding the partial differential equations and integrating over $\Omega$). It is possible to use this information, along with the uniform sup norm bound for $u_1$ to modify the proof of Theorem \ref{martinthm} to obtain a uniform sup norm bound for $u_2$.

\subsection* {Example 2} Here, we give an example related to the well known Brusselator. Consider the system
\begin{align}\label{brus}
u_{1_t}&=d_1 \Delta u_1 \quad & x\in \Omega, t>0  \nonumber\\ 
u_{2_t}&=d_2\Delta u_2 \quad & x\in \Omega, t>0 \nonumber\\
d_1\frac{\partial u_1}{\partial\eta}&=\alpha u_2-u_2^2u_1 \quad & x\in M, t>0\\\nonumber
d_2 \frac{\partial u_2}{\partial\eta}&=\beta-(\alpha+1)u_2+u_2^2u_1 \quad & x\in M, t>0\\\nonumber
u_i(x,0)&=w_i(x) & x\in \overline\Omega \nonumber
\end{align}
where $d_1, d_2, \alpha, \beta>0$ and $w$ is sufficiently smooth and componentwise nonnegative. If we define \[ F(u)=\begin{pmatrix} 0\\0\end{pmatrix} \quad \text{and } \quad G(u)=\begin{pmatrix} \alpha u_2-u_2^2u_1\\\beta-(\alpha+1)u_2+u_2^2u_1\end{pmatrix} \]for all $u\in\mathbb{R}_+^2$, then $(V_N)$, $(V_{F})$, $(V_{QP})$ and $(V_{Poly})$ are satisfied with $a_1\geq 1$ and $L_a=\max\lbrace \beta,\alpha\cdot a_1\rbrace$. Therefore, Theorem \ref{global} implies (\ref{brus}) has a unique, componentwise nonnegative, global solution.

\subsection*{Example 3} We next consider a general reaction mechanism of the form
\[ R_1+ R_2\substack{\longrightarrow\\ \longleftarrow} P_1\] 
where $R_i$ and $P_i$ represent reactant and product species, respectively. If we set $u_i=[R_i]$ for $i=1,2$, and $u_3 = [P_1] $, and let $k_f ,k_r$ be the (nonnegative) forward and reverse reaction rates, respectively, then we can model the process by the application of the law of conservation of mass and the second law of Fick (ﬂow) with the following reaction–diffusion system: 

\begin{align}\label{square1}
u_{i_t}&=d_i \Delta u_i \quad & x\in \Omega, t>0, i=1,2,3 \nonumber\\ \nonumber
d_1\frac{\partial u_1}{\partial\eta}&=-k_fu_1u_2+k_ru_3\quad & x\in M, t>0\\
d_2 \frac{\partial u_2}{\partial\eta}&=-k_fu_1u_2+k_ru_3\quad & x\in M, t>0\\\nonumber
d_3 \frac{\partial u_3}{\partial\eta}&=k_f u_1u_2-k_ru_3 \quad & x\in M, t>0\\\nonumber
u_i(x,0)&=w_i(x) & x\in \overline\Omega, i=1,23, \nonumber
\end{align}
where $d_i>0$ and the initial data $w$ is sufficiently smooth and componentwise nonnegative. If we define
\[ F(u)=\begin{pmatrix} 0\\0\\0\end{pmatrix} \quad \text{,} \quad G(u)=\begin{pmatrix}-k_fu_1u_2+k_rv_3\\ -k_fu_1u_2+k_rv_3\\k_f u_1u_2-k_rv_3   \end{pmatrix} \]
for all $u\in\mathbb{R}_+^3$, then $(V_N)$, $(V_{F})$, $(V_{QP})$ and $(V_{Poly})$ are satisfied. In addition, $(V_{L1})$ is satisfied with $L_1=0$ since \[ \frac{1}{2}H_1(z)+\frac{1}{2}H_2(z)+H_3(z)=0 \quad \text {and}\quad \frac{1}{2}F_1(z)+\frac{1}{2}F_2(z)+F_3(z)=0\] for all $z\in \mathbb{R}^3_{+}$. Therefore, the hypothesis of Theorems \ref{global} and \ref{uniform} are satisfied. As a result (\ref{square1}) has a unique, componentwise nonnegative, uniformly bounded, global solution. 

\subsection*{ Example 4} Finally, we consider a system that satisfies the hypothesis of the Theorem $\ref{global}$, where the boundary reaction vector field does not satisfy a linear intermediate sums condition. Let 

\begin{align}\label{square}
u_{1_t}&=d_1 \Delta u \quad & x\in \Omega, t>0  \nonumber\\ \nonumber
u_{2_t}&=d_2\Delta u \quad & x\in \Omega, t>0\\
d_1\frac{\partial u_1}{\partial\eta}&=\alpha u_1u_2^3-u_1u_2^2 \quad & x\in M, t>0\\\nonumber
d_2 \frac{\partial u_2}{\partial\eta}&=u_1u_2^2-\beta u_1u_2^6 \quad & x\in M, t>0\\\nonumber
u(x,0)&=w(x) & x\in \overline\Omega \nonumber
\end{align}
where $d_1, d_2, \alpha,\beta>0$ and $w$ is sufficiently smooth and componentwise nonnegative. In this setting 
\[ F(u)=\begin{pmatrix} 0\\0\end{pmatrix} \quad \text{,} \quad G(u)=\begin{pmatrix} \alpha u_1u_2^3-u_1u_2^2\\u_1u_2^2-\beta u_1u_2^6\end{pmatrix} \]
for all $u\in\mathbb{R}_+^2$. It is simple matter to see that $(V_N)$, $(V_{F})$, $(V_{QP})$ and $(V_{Poly})$ are satisfied. Also, if $a\ge 1$ then \[ a F_1(u)+F_2(u)=0 \quad \text {and}\quad aG_1(u)+G_2(u)\le (a\alpha-\beta)u_1(u_2^3-u_2^6)\le \frac{a\alpha}{4}u_1\] for all $u\in \mathbb{R}^2_{+}$. Consequenty, $(V_L)$ is satisfied. Therefore, Theorem \ref{global} implies $(\ref{square})$ has a unique, componentwise nonnegative, global solution. 
\section*{Acknowledgement} 
The author would like to thank Prof. Jeff Morgan for his comments, which lead to this work.

\end{document}